\documentclass[11pt]{article}
\usepackage{graphicx,indentfirst}
\usepackage{amscd}
\usepackage[all]{xy}
\usepackage{amsmath}
\usepackage{enumerate}
\usepackage{amsfonts,amssymb}
\usepackage{extarrows}
\usepackage{amsthm}
\usepackage{latexsym}
\usepackage[american]{babel}
\usepackage{times}
\usepackage{tikz}
\usepackage{tikz-cd}
\usetikzlibrary{arrows, decorations.markings}
\usepackage{url}
\usepackage{appendix}
\usepackage{bookmark}
\usepackage{verbatim}
\usepackage{hyperref}
\usepackage{MnSymbol}
\usepackage[affil-it]{authblk}

\usetikzlibrary{matrix,arrows,decorations.pathmorphing}
\usepackage[a4paper,left=2.5cm,right=2.5cm,bottom=3cm,top=3.5cm]{geometry}

\newtheorem{theorem}{Theorem}[section]
\newtheorem{lemma}[theorem]{Lemma}
\newtheorem{proposition}[theorem]{Proposition}
\newtheorem{corollary}[theorem]{Corollary}

\theoremstyle{definition}\newtheorem{definition}{Definition}[section]
\theoremstyle{remark}
\theoremstyle{remark}\newtheorem{rmk}{Remark}
\theoremstyle{remark}
\theoremstyle{remark}

\newcommand{\Hom}{\text{Hom}}

\newcommand{\id}{\text{id}}

\newcommand{\simto}{\overset{\sim}{\to}}
\newcommand{\rD}{\mathrm{D}}
\newcommand{\uC}{\underline{C}}
\newcommand{\ubI}{\underline{\mathbb{I}}}

\newcommand{\obj}{\text{obj}}

\begin{document}
\title{A recurrent formula of $A_{\infty}$-quasi inverses of dg-natural transformations between dg-lifts of derived functors}
\author{Zhaoting Wei}

\maketitle

\begin{abstract}
A dg-natural transformation between dg-functors is called an objectwise homotopy equivalence if its induced morphism on each object admits a homotopy inverse. In general an objectwise homotopy equivalence does not have a dg-inverse but has an $A_{\infty}$ quasi-inverse. In this note we give a recurrent formula of the $A_{\infty}$ quasi-inverse. This result is useful in studying the compositions of dg-lifts of derived functors of schemes.
\end{abstract}

\section{Introduction}
In \cite{schnurer2018six}, Schn\"{u}rer constructed Grothendieck six functor formalism of dg-enhancements for ringed spaces over a field $k$. In more details, for each $k$-ringed space $X$ we have a dg $k$-category $\ubI(X)$ which is a dg-enhancement of $\rD(X)$, the derived category of sheaves of $\mathcal{O}_X$-modules. Moreover for a morphism $f: X\to Y$ of $k$-ringed space we have dg $k$-functors  
$$
\underline{f}^*: \ubI(Y)\to \ubI(X) \text{ and } \underline{f}_*: \ubI(X)\to \ubI(Y)
$$
which are dg-lifts of the derived functos $\mathbf{L}f^*$ and $\mathbf{R}f_*$, respectively. In addition, Schn\"{u}rer showed that for two composable morphisms $f: X\to Y$ and $g: Y\to Z$, we have zig-zags of dg-natural transformations which are \emph{objectwise homotopy equivalences} between $\underline{(gf)}^*$ and $\underline{f}^*\underline{g}^*$ and 
 $\underline{(gf)}_*$ and $\underline{f}_*\underline{g}_*$.

We call a dg $k$-natural transformation $\Phi: F\rightarrow G$ an objectwise homotopy equivalence if for any object $\mathcal{E}$, the induced morphism 
$$
\Phi_{\mathcal{E}}: F(\mathcal{E})\to G(\mathcal{E})
$$
has a homotopy inverse. This does not mean that we could find a homotopy inverse of $\Phi$ because the objectwise homotopy inverses are not compatible with morphisms as illustrated in the following diagrams

$$
\begin{CD}
F(\mathcal{E}_1) @>\Phi_{\mathcal{E}_1}>>G(\mathcal{E}_1) \\
@VVF(\alpha)V  @V \circlearrowright ~G(\alpha) VV\\
F(\mathcal{E}_2) @>\Phi_{\mathcal{E}_2}>> G(\mathcal{E}_2)
\end{CD}    \text{ ~~~~~but }
\begin{CD}
F(\mathcal{E}_1) @<\Phi^{-1}_{\mathcal{E}_1}<<G(\mathcal{E}_1) \\
@VVF(\alpha)V  @ V \ncirclearrowright~G(\alpha) VV\\
F(\mathcal{E}_2)@<\Phi^{-1}_{\mathcal{E}_2}<< G(\mathcal{E}_2).
\end{CD} 
$$

Nevertheless, by \cite[Proposition 7.15]{lyubashenko2003category} we know that we can extend $\Phi^{-1}$ to an $A_{\infty}$ natural transformation $\Psi:G\Rightarrow F$ and $\Psi$ is an $A_{\infty}$ quasi-inverse of $\Phi$. In this note we give a detailed construction of the recurrent formula of $\Psi$ as suggested in \cite[Appendix B]{lyubashenko2003category}. In particular we show that we can construct $\Psi$ by compositions of objectwisely chosen homotopies. See Theorem \ref{thm: A-infty quasi-inverse} below. This formula will be used in \cite{wei2019twisted}.

\section*{Acknowledgments}
The author would like to thank Nick Gurski and Olaf Schn\"{u}rer for very helpful discussions.

\section{A review of dg-natural transformations and $A_{\infty}$-natural transformations}
In this section we review some concepts around dg-functors, dg-natural transformations, and $A_{\infty}$-natural transformations. 

\begin{rmk}
We would like to point out that the best way to describe $A_{\infty}$-categories\slash functors \slash natural-transformations is in the framework of bar constructions and dg-cocategories, see \cite{lyubashenko2003category}. In this note we just take the by-hand definition, which requires minimal amount of preparation but involves more complicated notations.
\end{rmk}

\begin{definition}[dg-categories]\label{def: dg-category}
Let $k$ be a commutative ring with unit. A \emph{differential graded} or \emph{dg $k$-category} is a category $\mathcal{C}$ whose morphism spaces are cochain complexes of $k$-modules and whose compositions of morphisms
$$
\mathcal{C}(Y,Z)\otimes_k \mathcal{C}(X,Y)\to  \mathcal{C}(X,Z)
$$
are morphisms of $k$-cochain complexes. Furthermore, there are obvious associativity and unit axioms.
\end{definition}

\begin{definition}[dg-functors]\label{def: dg-functor}
Let $k$ be a commutative ring with unit and $\mathcal{C}$ and $\mathcal{D}$ be two dg $k$-categories. A \emph{dg $k$-functor} $F: \mathcal{C}\to \mathcal{D}$ consists of the following data:
\begin{enumerate}
\item A map $F: \obj(\mathcal{C})\to \obj(\mathcal{D})$;
\item For any objects $X$, $Y\in \obj(\mathcal{C})$, a closed, degree $0$ morphism of complexes of $k$-modules
$$
F(X,Y):\mathcal{C}(X,Y)\to \mathcal{D}(FX,FY)
$$
which is compatible with the composition and the units.
\end{enumerate}
\end{definition}

\begin{definition}[dg-natural transformation]\label{def: dg-natural transformation}
Let $k$ be a commutative ring with unit and $F$, $G: \mathcal{C}\to \mathcal{D}$ be two dg $k$-functors between dg $k$-categories. A \emph{dg $k$-prenatural transformation} $\Phi: F\Rightarrow G$ of degree $n$ consists of a morphism
$$
\Phi_X\in  \mathcal{D}^n(FX,GX) \text{ for each object }X
$$
such that for any morphism $u\in \mathcal{C}^m(X,Y)$ we have
$$
\Phi_YFu=(-1)^{mn}Gu\Phi_X.
$$
The differential on $\Phi$ is defined objectwisely and it is clear that $d\Phi$ is  a dg $k$-prenatural transformation of degree $n+1$. We call $\Phi$ a \emph{dg $k$-natural transformation} if $\Phi$ is closed and of degree $0$.
\end{definition}

\begin{definition}[$A_{\infty}$-prenatural transformation]\label{def: A_infty-prenatural transformation}
Let $k$ be a commutative ring with unit and $F$, $G: \mathcal{C}\to \mathcal{D}$ be two dg $k$-functors between dg $k$-categories. An \emph{$A_{\infty}$ $k$-prenatural transformation} $\Phi: F\Rightarrow G$ of degree $n$ consists of the following data:
\begin{enumerate}
\item For any object $X\in \obj(\mathcal{C})$, a morphism $\Phi^0_X\in  \mathcal{D}^n(FX,GX)$;
\item For any $l\geq 1$ and any objects $X_0,\ldots, X_l\in \obj(\mathcal{C})$, a morphism 
$$
\Phi^l_{X_0,\ldots, X_l}\in \Hom^{n-l}_k(\mathcal{C}(X_{l-1},X_l)\otimes_k \ldots \otimes_k \mathcal{C}(X_0,X_1), \mathcal{D}(FX_0,GX_l))
$$
\end{enumerate}
\end{definition}

\begin{definition}[Differential of $A_{\infty}$-prenatural transformation]\label{def: differential of A_infty-prenatural transformation}
Let $k$ be a commutative ring with unit and $F$, $G: \mathcal{C}\to \mathcal{D}$ be two dg $k$-functors between dg $k$-categories. Let $\Phi: F\Rightarrow G$ be an $A_{\infty}$ $k$-prenatural transformation of degree $n$ as in Definition \ref{def: A_infty-prenatural transformation}. Then the \emph{differential} $d\Phi: F\Rightarrow G$ is an  $A_{\infty}$ $k$-prenatural transformation of degree $n+1$ whose components are given as follows:
\begin{enumerate}
\item For any object $X\in \obj(\mathcal{C})$,$(d^{\infty}\Phi)^0_X=d(\Phi^0_X) \in  \mathcal{D}^{n+1}(FX,GX)$;
\item For any $l\geq 1$ and a collection of morphisms $u_i\in  \mathcal{C}(X_{i-1},X_i)$ $i=1,\ldots, l$,
\begin{equation}\label{eq: differential of A infty prenatural transformation}
\begin{split}
&(d^{\infty}\Phi)^l(u_l\otimes\ldots \otimes u_1)=\\
&d(\Phi^l(u_l\otimes\ldots \otimes u_1))+(-1)^{|u_l|-1}G(u_l)\Phi^{l-1}(u_{l-1}\otimes \ldots \otimes u_1)\\
&+(-1)^{n|u_1|-|u_1|-\ldots-|u_l|+l-1}\Phi^{l-1}(u_l\otimes\ldots\otimes u_2)F(u_1)\\
&+\sum_{i=1}^l(-1)^{|u_l|+\ldots+|u_{i+1}|+l-i+1}\Phi^{l}(u_l\otimes\ldots du_i\otimes \ldots u_1)\\
&+\sum_{i=1}^{l-1}(-1)^{|u_l|+\ldots+|u_{i+1}|+l-i+1}\Phi^{l-1}(u_l\otimes\ldots u_{i+1}u_i\otimes \ldots u_1)
\end{split}
\end{equation}
\end{enumerate}
\end{definition}

We can check that $d^{\infty}\circ d^{\infty}=0$ on $A_{\infty}$ $k$-prenatural transformations.

\begin{definition}[$A_{\infty}$-natural transformation]\label{def: A_infty natural transformation}
Let $k$ be a commutative ring with unit and $F$, $G: \mathcal{C}\to \mathcal{D}$ be two dg $k$-functors between dg $k$-categories. Let $\Phi: F\Rightarrow G$ be an $A_{\infty}$ $k$-prenatural transformation. We call $\Phi$ an \emph{$A_{\infty}$ $k$-natural transformation} if $\Phi$ is of degree $0$ and closed under the differential $d^{\infty}$ in Definition \ref{def: differential of A_infty-prenatural transformation}.
\end{definition}

It is clear that a dg $k$-natural transformation $\Phi$ can be considered as an $A_{\infty}$ $k$-natural transformation  with $\Phi^l=0$ for all $l\geq 1$.

\begin{definition}[Compositions]\label{def: composition of A-infty natural transformation}
Let $k$ be a commutative ring with unit and $F$, $G$, $H: \mathcal{C}\to \mathcal{D}$ be three dg $k$-functors between dg $k$-categories. Let $\Phi: F\Rightarrow G$ be a dg $k$-natural transformation and $\Psi: G\Rightarrow H$ be an $A_{\infty}$ $k$-natural transformation. Then the composition $\Psi\circ \Phi$ is defined as follows: For any object $X\in \obj(\mathcal{C})$
$$
(\Psi\circ \Phi)^0_X:=\Psi^0_X\Phi_X: FX\to GX\to HX
$$
and for any $u_i\in \mathcal{C}(X_{i-1},X_i)$, $i=1,\ldots, l$
$$
(\Psi\circ \Phi)^l(u_l\otimes\ldots u_1):=\Psi^l(u_l\otimes\ldots u_1) \Phi_{X_0}
$$
We can check that $\Psi\circ \Phi$ is an $A_{\infty}$ $k$-natural transformation.

Similarly,  Let $\Phi: F\Rightarrow G$ be an $A_{\infty}$ $k$-natural transformation and $\Psi: G\Rightarrow H$ be a dg $k$-natural transformation. Then  the composition $\Psi\circ \Phi$ is defined as follows: For any object $X\in \obj(\mathcal{C})$
$$
(\Psi\circ \Phi)^0_X:=\Psi_X\Phi^0_X: FX\to GX\to HX
$$
and for any $u_i\in \mathcal{C}(X_{i-1},X_i)$, $i=1,\ldots, l$
$$
(\Psi\circ \Phi)^l(u_l\otimes\ldots u_1):=\Psi_{X_l} \Phi^l(u_l\otimes\ldots u_1)
$$
We can check that $\Psi\circ \Phi$ is an $A_{\infty}$ $k$-natural transformation.
\end{definition}

\begin{rmk}
We can define compositions for general $A_{\infty}$ $k$-prenatural transformations. See \cite[Section 3]{lyubashenko2003category} or \cite[Section I.1(d)]{seidel2008fukaya}.
\end{rmk}

\begin{definition}[$A_{\infty}$ quasi-inverse]\label{def: A-infty quasi-inverse}
Let $k$ be a commutative ring with unit and $F$, $G: \mathcal{C}\to \mathcal{D}$ be two dg $k$-functors between dg $k$-categories. Let $\Phi: F\Rightarrow G$ be a dg $k$-natural transformation. We call an  $A_{\infty}$ $k$-natural transformation $\Psi: G\Rightarrow F$ an \emph{$A_{\infty}$ quasi-inverse} of $\Phi$ if there exists $A_{\infty}$ $k$-prenatural transformations $\eta: F\Rightarrow F$ and $\omega: G\Rightarrow G$ both of degree $-1$ such that
$$
\Psi\circ \Phi-\id_F=d^{\infty}\eta, \text{ and } \Phi\circ \Psi-\id_G=d^{\infty}\omega.
$$
In more details, this means that we have
$$
\Psi^0_X\Phi_X-\id_{FX}=d\eta^0_X, \text{ and } \Phi_X\Psi^0_X-\id_{GX}=d\omega^0_X \text{ for any } X\in \obj{\mathcal{C}}
$$
and for any $l\geq 1$ and any $u_i\in \mathcal{C}(X_{i-1},X_i)$, $i=1,\ldots, l$, we have
\begin{equation}
\begin{split}
&\Psi^l(\Phi(u_l)\otimes\ldots\otimes \Phi(u_1))=\\
&d(\eta^l(u_l\otimes\ldots \otimes u_1))+(-1)^{|u_l|-1}G(u_l)\eta^{l-1}(u_{l-1}\otimes \ldots \otimes u_1)\\
&+(-1)^{-|u_2|-\ldots-|u_l|+l-1}\eta^{l-1}(u_l\otimes\ldots\otimes u_2)F(u_1)\\
&+\sum_{i=1}^l(-1)^{|u_l|+\ldots+|u_{i+1}|+l-i+1}\eta^{l}(u_l\otimes\ldots du_i\otimes \ldots u_1)\\
&+\sum_{i=1}^{l-1}(-1)^{|u_l|+\ldots+|u_{i+1}|+l-i+1}\eta^{l-1}(u_l\otimes\ldots u_{i+1}u_i\otimes \ldots u_1)
\end{split}
\end{equation}
and
\begin{equation}
\begin{split}
&\Phi(\Psi^l(u_l\otimes\ldots\otimes u_1))=\\
&d(\omega^l(u_l\otimes\ldots \otimes u_1))+(-1)^{|u_l|-1}G(u_l)\omega^{l-1}(u_{l-1}\otimes \ldots \otimes u_1)\\
&+(-1)^{-|u_2|-\ldots-|u_l|+l-1}\omega^{l-1}(u_l\otimes\ldots\otimes u_2)F(u_1)\\
&+\sum_{i=1}^l(-1)^{|u_l|+\ldots+|u_{i+1}|+l-i+1}\omega^{l}(u_l\otimes\ldots du_i\otimes \ldots u_1)\\
&+\sum_{i=1}^{l-1}(-1)^{|u_l|+\ldots+|u_{i+1}|+l-i+1}\omega^{l-1}(u_l\otimes\ldots u_{i+1}u_i\otimes \ldots u_1)
\end{split}
\end{equation}
\end{definition}

\section{Review of functorial injective resolutions and a dg-lifts of derived functors}\label{section: dg-lift}
The main reference of this section is \cite{schnurer2018six}.
\subsection{Functorial injective resolutions}
Let $k$ be a field and $X$ be a $k$-ringed space. Let $\uC(X)$ be the dg $k$-category of complexes of sheaves on $X$ and $\ubI(X)$ its full dg $k$-subcategory of h-injective complexes of injective sheaves. Let $\ubI^b(X)$ and $\ubI^+(X)$ be the full subcategories of $\ubI(X)$ consisting of complexes with bounded or bounded below cohomology sheaves, respectively. See \cite{spaltenstein1988resolutions} or \cite[Chapter 14]{kashiwara2005categories} for an introduction to h-injective complexes.

 It is clear that  $\ubI(X)$ is a strongly pretriangulated dg $k$-category hence its homotopy category  $[\ubI(X)]$ is a triangulated $k$-category and the obvious functor $[\ubI(X)]\to \rD(X)$ is a triangulated equivalence.

We could construct an equivalence in the other direction.

\begin{proposition}\label{prop: functorial injective resolution}[\cite[Corollary 2.3]{schnurer2018six}]
Let $k$ be a field. Let $(X,\mathcal{O})$ be a $k$-ringed site and let $\uC(X)_{\text{hflat, cwflat}}$ denote the full dg $k$-subcategory of  $\uC(X)$ of h-flat and componentwise flat objects. Then there exists  dg $k$-functors
\begin{equation}\label{eq: functorial injective resolutions}
\begin{split}
\mathbf{i}:& \uC(X)\to \ubI(X)\\
\mathbf{e}:& \uC(X)\to \uC(X)_{\text{hflat, cwflat}}
\end{split}
\end{equation}

together with   dg $k$-natural transformations
\begin{equation}\label{eq: functorial injective natural transformations}
\begin{split}
\iota:& \id\to \mathbf{i}: \uC(X)\to \uC(X)\\
\epsilon:& \mathbf{e} \to \id: \uC(X)\to \uC(X)
\end{split}
\end{equation}
whose evaluations $\iota_{\mathcal{F}}:\mathcal{F}\to \mathbf{i}\mathcal{F}$ and $\epsilon_{\mathcal{F}}:\mathbf{e}\mathcal{F}\to \mathcal{F}$ at each object $\mathcal{F}\in \uC(X)$ are quasi-isomorphisms. 
\end{proposition}

It is clear that the induced functor $[\mathbf{i}]: [\uC(X)]\to [\ubI(X)]$ sends acyclic objects to zero, hence it factors to an equivalence
$$
\bar{[\mathbf{i}]}:\rD(X)\simto [\ubI(X)]
$$
of triangulated $k$-categories.

Intuitively the  dg $k$-functor $\mathbf{i}$ in Proposition \ref{prop: functorial injective resolution} could be considered as a functorial injective resolution.

\begin{rmk}
The result in Proposition \ref{prop: functorial injective resolution} is an adaption of general results from enriched model category theory in \cite{riehl2014categorical}, in particular \cite[Corollary 13.2.4]{riehl2014categorical}.
\end{rmk}

\begin{rmk}
The assumption that $k$ is a field is essential for Proposition \ref{prop: functorial injective resolution}. Actually if $k=\mathbb{Z}$, then the pair $(\mathbf{i},\iota)$ in Proposition \ref{prop: functorial injective resolution} does not exist. See \cite[Lemma 4.4]{schnurer2018six} for a counterexample.
\end{rmk}

\subsection{A dg-lift of the pull-back functor and push-forward functor}
\begin{definition}\label{def: injective pull back and push forward functors}[\cite[2.3.4]{schnurer2018six}]
Let $k$ be a field. For a morphism of $k$-ringed spaces $f: X\to Y$, we define the injective pull back dg $k$-functor $\underline{f}^*$ as 
\begin{equation}\label{eq: injective pull back functor}
\underline{f}^*:=\mathbf{i}\circ f^*\circ \mathbf{e}
\end{equation}
Similarly we define the injective push forward dg $k$-functor $\underline{f}_*$ as
\begin{equation}\label{eq: injective push forward functor}
\underline{f}_*:=\mathbf{i}\circ f_* 
\end{equation}
where $\mathbf{i}$ and $\mathbf{e}$ are defined in Proposition \ref{prop: functorial injective resolution}.
\end{definition}

\begin{rmk}
Actually in \cite{schnurer2018six} all Grothendieck's six functors were lifted to dg $k$-functors.
\end{rmk}

\begin{proposition}\label{prop: injective pull back functor lifts derived pull back functor}[\cite[Proposition 6.5]{schnurer2018six}]
Let $k$ be a field. For a morphism of $k$-ringed spaces $f: X\to Y$, the dg $k$-functors $\underline{f}^*$ and $\underline{f}_*$ in Definition \ref{def: injective pull back and push forward functors} are dg-lifts of the derived pull back and derived push forward functors $\mathbf{L}f^*: \rD(Y)\to \rD(X)$ and $\mathbf{R}f_*: \rD(X)\to \rD(Y)$, respectively. More precisely, the diagrams
$$
\begin{CD}
\rD(Y) @>\mathbf{L}f^*>> \rD(X)\\
@V\bar{[\mathbf{i}]}V\sim V @V\bar{[\mathbf{i}]}V\sim V\\
 [\ubI(Y)] @>[\underline{f}^*]>>  [\ubI(X)]
\end{CD}
$$
and
$$
\begin{CD}
\rD(X) @>\mathbf{R}f_*>> \rD(Y)\\
@V\bar{[\mathbf{i}]}V\sim V @V\bar{[\mathbf{i}]}V\sim V\\
 [\ubI(X)] @>[\underline{f}_*]>>  [\ubI(Y)]
\end{CD}
$$
commute up to a canonical $2$-isomorphism.
\end{proposition}

\subsection{Objectwise homotopy equivalences}
By Definition \ref{def: injective pull back and push forward functors} it is clear that we do not have $\underline{(gf)}^*=\underline{f}^*\underline{g}^*$. Actually  $\underline{(gf)}^*$ and $\underline{f}^*\underline{g}^*$ are connected by a zig-zag of dg natural transformations. To describe this relation more clearly, we introduce the following definitions.

\begin{definition}\label{def: objectwise homotopy equivalence}[\cite[2.1.3]{schnurer2018six}]
Let $k$ be a field and $X$, $Y$ be $k$-ringed spaces.  Let  $F,G: \ubI(X) \to \ubI(Y)$ be dg $k$-functors. A dg $k$-natural transformation $\Phi: F\Rightarrow G$ is called an \emph{objectwise homotopy equivalence} if for any object $\mathcal{E}\in \obj(\ubI(X))$, the morphism $\Phi_{\mathcal{E}}: F\mathcal{E}\to G\mathcal{E}$ has a homotopic inverse.
\end{definition}

\begin{proposition}\label{prop: composition of pull back functors and push forward functors}[\cite[Proposition 6.17, Lemma 6.21]{schnurer2018six}]
Let $k$ be a field and $X\overset{f}{\to}Y\overset{g}{\to} Z$ be morphisms of $k$-ringed spaces. Then there exist zig-zags of objectwise homotopy equivalences
\begin{equation}\label{eq: 2 morphisms for compositions}
\begin{split}
T^{f,g}&: \underline{(gf)}^*\simto \underline{f}^*\underline{g}^*\\
T_{f,g}&: \underline{(gf)}_*\simto \underline{g}_*\underline{f}_*\\
T^{\id}&: \underline{\id}^*\simto \id\\
T_{\id}&: \underline{\id}_*\simto \id\\
\end{split}
\end{equation}
\end{proposition}
\begin{proof}
We give the relation between $\underline{(gf)}^*$ and $\underline{f}^*\underline{g}^*$ to illustrate the idea. We use the dg $k$-natural transformations $\iota: \id\to \mathbf{i}$ and $\epsilon: \mathbf{e}\to \id$ in Proposition \ref{prop: functorial injective resolution} and have the following objectwise homotopy equivalences
\begin{equation*}
\begin{split}
 \underline{(gf)}^*&= \mathbf{i}(gf)^*\mathbf{e}\simto \mathbf{i}f^*g^*\mathbf{e}\\
& \xleftarrow[\sim]{\epsilon}  \mathbf{i}f^*\mathbf{e}g^*\mathbf{e}\xrightarrow[\sim]{\iota}  \mathbf{i}f^*\mathbf{e}\mathbf{i}g^*\mathbf{e}\\
&=\underline{f}^*\underline{g}^*.
\end{split}
\end{equation*}
\end{proof}

\section{Objectwise homotopy equivalences and $A_{\infty}$ quasi-inverses}

\begin{definition}\label{def: homotopy inverse systems}
Let $k$ be a field and $X$, $Y$ be   $k$-ringed spaces. Let   $F$, $G: \ubI(X)\to \ubI(Y)$ be two dg $k$-functors and $\Phi: F\to G$ be a dg $k$-natural transformation which is an objectwise homotopy equivalence. For each object $\mathcal{E}\in \obj(\ubI(X))$ we can choose $\Psi_{\mathcal{E}}\in \ubI^0(Y)(G\mathcal{E},F\mathcal{E})$, $h_{\mathcal{E}}\in \ubI^{-1}(Y)(F\mathcal{E},F\mathcal{E})$, $p_{\mathcal{E}}\in \ubI^{-1}(Y)(G\mathcal{E},G\mathcal{E})$, 
such that
$$
\Psi_{\mathcal{E}}\Phi_{\mathcal{E}}-\id_{F\mathcal{E}}=dh_{\mathcal{E}},~\Phi_{\mathcal{E}}\Psi_{\mathcal{E}}-\id_{G\mathcal{E}}=dp_{\mathcal{E}}.
$$
 We call such a choice an \emph{objectwise homotopy inverse system} of $\Phi$. 
\end{definition}

For a objectwise homotopy equivalence $\Phi$, its homotopy invese system always exists.

The following theorem is the main result of this note.

\begin{theorem}\label{thm: A-infty quasi-inverse}
Let $k$ be a field and $X$, $Y$ be   $k$-ringed spaces. Let   $F$, $G: \ubI(X)\to \ubI(Y)$ be two dg $k$-functors and $\Phi: F\to G$ be a dg $k$-natural transformation which is an objectwise homotopy equivalence. Then there exists an $A_{\infty}$ quasi-inverse of $\Phi$. More precisely, we choose and fix an objectwise homotopy inverse system $\mathcal{H}$ of $\Phi$ as in Definition \ref{thm: A-infty quasi-inverse} and there exist an $A_{\infty}$ $k$-natural transformation $\Psi: G\Rightarrow F$ and $A_{\infty}$ $k$-prenatural transformations  $\eta: F\Rightarrow F$ and $\omega: G\Rightarrow G$ of degree $-1$ such that
$$
\Psi\circ \Phi-\id_F=d\eta, \text{ and } \Phi\circ \Psi-\id_G=d\omega.
$$
Moreover, $\Psi$, $\eta$, and $\omega$ are defined by compositions of $F$, $G$, $\Phi$, and $\mathcal{H}$.
\end{theorem}
\begin{proof}
The proof is a refinement of \cite[Proposition 7.15]{lyubashenko2003category}. We construct $\Psi$, $\eta$, and $\omega$ by induction. First we construct the left inverse. Let $\Psi^0_{\mathcal{E}}=\Psi_{\mathcal{E}}$ and $\eta^0_{\mathcal{E}}=p_{\mathcal{E}}$ as in Definition \ref{def: homotopy inverse systems}. Now suppose that for an $m\geq 1$ we have constructed $\Psi^i$ and $\eta^i$,  $i=1,\ldots, m-1$ by compositions of  $F$, $G$, $\Phi$, and $\mathcal{H}$ such that
the auxiliary $A_{\infty}$ $k$-prenatural transformations 
\begin{equation*}
\begin{split}
\widetilde{\Psi}=&(\Psi^0,\Psi^1,\ldots, \Psi^{m-1},0,\ldots)\\
\widetilde{\eta}=&(\eta^0,\eta^1,\ldots, \eta^{m-1},0,\ldots)\\
\end{split}
\end{equation*}
satisfy
$$
(d^{\infty}\widetilde{\Psi})^l=0, \text{ and }  (\widetilde{\Psi}\circ \Phi-\id_F)^l=(d^{\infty}\widetilde{\eta})^l 
$$
hold for  $l=1,\ldots, m-1$. Now we introduce

We denote $(d^{\infty}\widetilde{\Psi})^m$ by $\lambda^m$.  For objects $\mathcal{E}_0,\ldots,\mathcal{E}_m$, $\lambda_m$ can by considered as a degree $1-m$ map 
$$ \ubI(X)(\mathcal{E}_{m-1},\mathcal{E}_m)\otimes\ldots \otimes  \ubI(X)(\mathcal{E}_0,\mathcal{E}_1)\to  \ubI(Y)(G\mathcal{E}_0,F\mathcal{E}_m).
$$
For later applications we consider $\lambda^m$ as a degree $0$ map
$$
\lambda^m\in \Hom^0( \ubI(X)(\mathcal{E}_{m-1},\mathcal{E}_m)[1]\otimes\ldots \otimes  \ubI(X)(\mathcal{E}_0,\mathcal{E}_1)[1],  \ubI(Y)(G\mathcal{E}_0,F\mathcal{E}_m)[1])
$$

Moreover we denote $(\widetilde{\Psi}\circ  \Phi-\id_F-d^{\infty}\widetilde{\eta})^m$ by $\mu^m$. As before we have
$$
\mu^m \in \Hom^{-1}( \ubI(X)(\mathcal{E}_{m-1},\mathcal{E}_m)[1]\otimes\ldots \otimes  \ubI(X)(\mathcal{E}_0,\mathcal{E}_1)[1],  \ubI(Y)(F\mathcal{E}_0,F\mathcal{E}_m)[1])
$$

\begin{lemma}\label{lemma: lambda defined by compositions}
In the above notation, $\lambda^m$ and $\mu^m$ are defined by by compositions of  $F$, $G$, $\Phi$, and $\mathcal{H}$.
\end{lemma}
\begin{proof}
By Definition \ref{def: differential of A_infty-prenatural transformation} we have
\begin{equation*}
\begin{split}
\lambda^m&(u_m\otimes\ldots\otimes u_1)=(d^{\infty}\widetilde{\Psi})^m(u_m\otimes\ldots\otimes u_1)=\\
&(-1)^{|u_m|-1}G(u_l)\Psi^{m-1}(u_{m-1}\otimes \ldots \otimes u_1)\\
&+(-1)^{n|u_1|-|u_1|-\ldots-|u_m|+m-1}\Psi^{m-1}(u_m\otimes\ldots\otimes u_2)F(u_1)\\
&+\sum_{i=1}^{m-1}(-1)^{|u_m|+\ldots+|u_{i+1}|+m-i+1}\Psi^{m-1}(u_m\otimes\ldots u_{i+1}u_i\otimes \ldots u_1)
\end{split}
\end{equation*}
The claim for $\lambda^m$ is clear by the induction hypothesis. We can prove the claim for $\mu^m$ in the same way.
\end{proof}

\begin{lemma}\label{lemma: cone}
In the above notation, we have
$$
d\lambda^m=0, \text{ and } d\mu^m=\lambda^m\Phi_{\mathcal{E}_0}.
$$
\end{lemma}
\begin{proof}
By the induction hypothesis we have $(d^{\infty}\widetilde{\Psi})^l=0$ for $l=1,\ldots, m-1$. Then  $d^{\infty}d^{\infty}\widetilde{\Psi}=d[(d^{\infty}\widetilde{\Psi})^m]=d\lambda^m$. Bu we have  $d^{\infty}d^{\infty}\widetilde{\Psi}=0$ hence $d\lambda^m=0$.

Since $(\widetilde{\Psi}\circ \Phi-\id_F-d^{\infty}\widetilde{\eta})^l=0$ for $l=1,\ldots, m-1$, it is clear that
$$
[d^{\infty}(\widetilde{\Psi}\circ \Phi-\id_F-d^{\infty}\widetilde{\eta})]^m=d\mu^m
$$
On the other hand since $\Phi$ and $\id_F$ are dg $k$-natural transformations, we have $d^{\infty}\Phi=0$ and $d^{\infty}\id_F=0$. Therefore
$$
d^{\infty}(\widetilde{\Psi}\circ \Phi-\id_F-d^{\infty}\widetilde{\eta})= (d^{\infty}\widetilde{\Psi})\circ \Phi.
$$
Compare the degree $m$ component we have $d\mu^m=\lambda^m\Phi_{\mathcal{E}_0}$.
\end{proof}

As suggested by \cite[Appendix B]{lyubashenko2003category} we let 
\begin{equation}\label{eq: induction construction lambda}
\Psi^m=\lambda^mp_{\mathcal{E}_0}-\mu^m\Psi_{\mathcal{E}_0}
\end{equation}
and
\begin{equation}\label{eq: induction construction mu}
\begin{split}
& \eta^m=-\mu^mh_{\mathcal{E}_0}+\mu^m\Psi_{\mathcal{E}_0}\Phi_{\mathcal{E}_0}h_{\mathcal{E}_0}\\
-\lambda^mp_{\mathcal{E}_0}\Phi_{\mathcal{E}_0}h_{\mathcal{E}_0}&-\mu^m\Psi_{\mathcal{E}_0}p_{\mathcal{E}_0}\Phi_{\mathcal{E}_0}+\lambda^mp_{\mathcal{E}_0}p_{\mathcal{E}_0}p_{\mathcal{E}_0}\Phi_{\mathcal{E}_0}.
\end{split}
\end{equation}
It is clear that 
\begin{equation}\label{eq: induction result}
d(\Psi^m)=-\lambda^m, \text{ and } d(\eta^m)=\Psi^m\Phi_{\mathcal{E}_0}+\mu^m.
\end{equation}
Let
\begin{equation*}
\begin{split}
\widetilde{\widetilde{\Psi}}=&(\Psi^0,\Psi^1,\ldots, \Psi^{m-1},\Psi^m,0,\ldots)\\
\widetilde{\widetilde{\eta}}=&(\eta^0,\eta^1,\ldots, \eta^{m-1},\eta^m,0,\ldots)\\
\end{split}
\end{equation*}
It is clear by Equation \eqref{eq: induction result} and the induction hypothesis that 
$$
(d^{\infty}\widetilde{\widetilde{\Psi}})^l=0, \text{ and }  (\widetilde{\widetilde{\Psi}}\circ \Phi-\id_F)^l=(d^{\infty}\widetilde{\widetilde{\eta}})^l 
$$
hold for  $l=1,\ldots, m$. Then by induction we construct an $A_{\infty}$ $k$-natural transformation $\Psi: G\Rightarrow F$  of degree $0$ and an $A_{\infty}$ $k$-prenatural transformation $\eta: F\Rightarrow F$ of degree $-1$ such that
$$
d^{\infty}\Psi=0, \text{ and }  \Psi\circ \Phi-\id_F=d^{\infty} \eta.
$$
Notice that $\Psi$ and $\eta$ are defined by compositions of $F$, $G$, $\Phi$, and $\mathcal{H}$.

We need to construct the homotopy of the other composition. Actually in the same way we can construct an $A_{\infty}$ $k$-natural transformation $\Psi^{\prime}: G\Rightarrow F$  of degree $0$ and an $A_{\infty}$ $k$-prenatural transformation $\omega^{\prime}: G\Rightarrow G$ of degree $-1$ such that
$$
d^{\infty}\Psi^{\prime}=0, \text{ and }  \Phi\circ \Psi^{\prime}-\id_G=d^{\infty} \omega^{\prime}
$$
where  $\Psi^{\prime}$ and  $\omega^{\prime}$  are also defined by compositions of $F$, $G$, $\Phi$, and $\mathcal{H}$.
It is clear that
$$
\Psi^{\prime}=\Psi+d^{\infty}(\Psi\omega^{\prime}-\eta\Psi^{\prime}).
$$
Therefore let
$$
\omega:=\omega^{\prime}+\Phi\eta\Psi^{\prime}-\Phi\Psi\omega^{\prime}
$$
then we have 
$$
\Phi\circ \Psi-\id_G=d^{\infty} \omega.
$$
and $\omega$ is also defined by compositions of $F$, $G$, $\Phi$, and $\mathcal{H}$.
\end{proof}

\begin{rmk}
The recurrent definition of $\Psi$ and $\eta$ is given by Equation \eqref{eq: induction construction lambda} and \eqref{eq: induction construction mu}.
\end{rmk}

\begin{corollary}
Let $k$ be a field and $X\overset{f}{\to}Y\overset{g}{\to} Z$ be morphisms of $k$-ringed spaces. Then there exist $A_{\infty}$ $k$-natural transformations
\begin{equation}\label{eq: 2 morphisms for compositions}
\begin{split}
T^{f,g}&: \underline{(gf)}^*\simto \underline{f}^*\underline{g}^*\\
T_{f,g}&: \underline{(gf)}_*\simto \underline{g}_*\underline{f}_*\\
T^{\id}&: \underline{\id}^*\simto \id\\
T_{\id}&: \underline{\id}_*\simto \id\\
\end{split}
\end{equation}
which only depend on the choice of objectwise homotopy inverse systems as in Definition \ref{def: homotopy inverse systems}.
\end{corollary}
\begin{proof}
It is a direct corollary of Proposition \ref{prop: composition of pull back functors and push forward functors} and Theorem \ref{thm: A-infty quasi-inverse}.
\end{proof}

\begin{rmk}
In \cite{wei2019twisted} the formula in this note is used to construct $A_{\infty}$-inverse of dg-natural transformations between twisted complexes. Moreover in an upcoming work \cite{wei2019homotopy}, the formula in this note, together with the method in \cite{wei2016twisted}, \cite{wei2018descent}, \cite{block2017explicit}, and \cite{arkhipov2018homotopy2}, can be used to obtain an injective dg-resolution of the equivariant derived category \cite{bernstein1994equivariant}. 
\end{rmk}

\bibliography{Ainfinity}{}
\bibliographystyle{alpha}
\end{document}